\documentclass{amsart}

\usepackage{amssymb}
\usepackage{amsmath}
\usepackage{amsthm}
\usepackage{amsbsy}
\usepackage{bm}
\usepackage{tikz}

\theoremstyle{plain}
\newtheorem{theorem}{Theorem}[section]
\newtheorem{lemma}[theorem]{Lemma}

\theoremstyle{definition}
\newtheorem{definition}[theorem]{Definition}

\theoremstyle{remark}

\theoremstyle{remark}
\newtheorem{remark}[theorem]{Remark}

\theoremstyle{remark}

\newtheorem*{acknowledgements}{Acknowledgements}
\newtheorem*{question}{Question}

\newcommand{\R}{\mathbb{R}}

\begin{document}
\title{Graphs of Systoles on  hyperbolic surfaces}

\author{Bidyut Sanki}

\email{bidyut.iitk7@gmail.com}

\author{Siddhartha Gadgil}

\email{gadgil@math.iisc.ernet.in}

\address{   Department of Mathematics,\\
        Indian Institute of Science,\\
        Bangalore 560012, India}

\date{\today}


\begin{abstract}

Given a hyperbolic surface, the set of all closed geodesics whose length is minimal form a graph on the surface, in fact a so-called fat graph, which we call the systolic graph. We study which fat graphs are systolic graphs for some surface (we call these admissible).

There is a natural necessary condition on such graphs, which we call combinatorial admissibility. Our first main result is that this condition is also sufficient.

It follows that a sub-graph of an admissible graph is admissible. Our second major result  is that there are infinitely many minimal non-admissible fat graphs (in contrast, for instance, to the classical result that there are only two minimal non-planar graphs).
\end{abstract}

\maketitle


\section{Introduction}

Given a hyperbolic surface $F$, each homotopy class of closed curve has a unique geodesic representative ~\cite{CB}. The lengths of the closed geodesics form the so-called \emph{length spectrum}, and the minimum of these lengths is the \emph{systole} which is denoted by $Sys(F)$.

We shall call the union of all closed geodesics whose length is the systole the \emph{systolic graph} associated to a surface. This is in fact a so-called fat graph, with all nodes having valence even and at least $4$. Henceforth when we refer to fat graphs we always assume that this valence condition is satisfied.

The central question of this paper is the following:

\begin{question}
 Which fat graphs are systolic graphs of hyperbolic surfaces?
\end{question}

Besides its relation to the study of systolic geometry and length spectra, we are motivated to study this question as we get a natural decomposition of the moduli space of hyperbolic surfaces by associating to a surface its systolic graph.

We call a fat graph admissible if it is the systolic graph of a hyperbolic surface, so that no complementary region is a disc (i.e., the graph is essential). Thus, our central goal is to understand which fat graphs are admissible.

We only consider systolic graphs that are essential, in the sense that no complementary region is homeomorphic to a disc. We remark that at critical points of the systole, all the complementary regions are homeomorphic to discs.

\begin{acknowledgements}
We thank the referee for several helpful comments and suggestions.  
\end{acknowledgements}

\subsection*{Combinatorial formulation}

An essential systolic graph $\Gamma$ of a hyperbolic surface $F$ can be viewed as a metric graph, with distance obtained by measuring along paths in the graph using the metric from the surface. The minimal geodesics are cycles in this graph, all of which have the same length, namely the systole $Sys(F)$. Further, any other cycle $\lambda$ in $\Gamma$ which gives an essential closed curve in the surface $F$, which is hence homotopic to a geodesic of length greater than $Sys(F)$. It follows that the length of $\lambda$ is greater than that of the systole.

Note that cycles of $\Gamma$ which correspond to minimal geodesics of $F$ can be determined from the fat graph $\Gamma$ (which we assume satisfies the valence conditions) -- we call these the standard cycles of $\Gamma$. Thus, we can formulate a necessary condition for $\Gamma$ to be admissible in terms of metric graph structures on $\Gamma$. Namely, if $\Gamma$ is admissible then we can associate lengths to the edges of $\Gamma$ so that
\begin{enumerate}
 \item All standard cycles have the same length, say $\sigma$.
 \item All other cycles have length greater than $\sigma$.
\end{enumerate}

We say that a graph is combinatorially admissible if we can associate lengths to edges satisfying the above condition. Our first main result says that this condition in fact characterizes admissibility.

\begin{theorem}\label{T:admissible}
A fat graph $\Gamma$ is admissible if and only if it is combinatorially admissible.
\end{theorem}

The proof of this result is based on negative curvature of hyperbolic space. The crucial ingredient is that if the systolic length is very large, then the lengths of cycles in a systolic graph in the metric on the graph are very close to the lengths of the corresponding geodesics on a hyperbolic surface.

\subsection*{Minimal obstructions}

Given a fat graph $\Gamma$, we can associate to it subgraphs that are unions of some of the standard cycles of $\Gamma$. It is easy to see that if $\Gamma$ is admissible, each such subgraph is combinatorially admissible, hence is admissible. Thus, it suffices to understand which fat graphs are \emph{minimally non-admissible}, i.e., which are non-admissible but with all proper subgraphs admissible.

This is a common situation in graph theory -- for instance planarity is similarly characterized by describing the minimally non-planar graphs, namely $K_{3,3}$ and $K_5$. However, in contrast to the simple answer in that case, we see that the complexity of the question we are studying in the following result.

\begin{theorem}\label{thm:1.2}
There are infinitely many minimally non-admissible fat graphs.
\end{theorem}

The configurations of systolic curves on hyperbolic surfaces has been studied extensively in ~\cite{APA},~\cite{PB},~\cite{FH},~\cite{PS},~\cite{PS1} and~\cite{PSS}. In \cite{PB}, Buser has shown that for all $l$ and $g\geq e^{\frac{l^2}{4}}$ there exists a hyperbolic surface $F_g$ of genus $g$ such that $Sys(F_g)=l$ which solves the large systole problem. Buser and Sarnak~\cite{PS} were the first to show that there exist families $F_{g_k}$ of closed hyperbolic surfaces of genus $g_k$ with $g_k\to \infty$ as $k\to \infty$ whose systole length grows like $Sys(F_{g_k}) \geq \frac{4}{3}\log g_k$. The notion of topological Morse function was introduced by Morse himself. Paul Schmutz Schaller~\cite{PSS} initiated the study of critical points of the topological Morse function $Sys$. In~\cite{Akrout}, Akrout showed that the systole function is a topological Morse function.


\section{Decorated fat graph}
In this section, we define decorated fat graph and its standard cycle. We start by recalling a definition of graph.

\begin{definition}
A graph is a quadruple $G=(V, H, s, i)$ where
\begin{enumerate}

\item $V$ is a non-empty set, called the set of vertices or nodes.

\item $H$ is a set (possibly empty), called the set of half edges.

\item $s:H\to V$ is a function, thought of as sending each half edge to the node which it contains.

\item $i:H\to H$ is a fixed point free involution map, thought of as sending each half edge to its other half.

\end{enumerate}
\end{definition}
A cycle in a graph is a simple closed path.
\begin{definition}
The girth $T(G)$ of a graph $G$ is the length of a shortest non-trivial cycle. If a graph does not contain any cycle (i.e., it is a tree), its girth is defined to be infinity.
\end{definition}

\begin{definition}
A fat graph is a graph $(V, E, s, i)$ with a bijection $\sigma:H \longrightarrow H$ whose cycles correspond to the sets $s^{-1}(v)$, $v \in V.$
\end{definition}

\begin{definition}
A decorated fat graph is a fat graph together with the union of disjoint circles (possibly empty) such that the degree of each node is even and at least 4.
\end{definition}

In particular, a disjoint union of topological circles is considered as a decorated fat graph. In this paper by a fat graph we always mean a decorated fat graph.

\begin{definition}
A simple cycle is called a standard cycle if every two consecutive edges are opposite to each other in the cyclic ordering on the set of edges incident at their common node. If a cycle is not standard, we call the cycle as non-standard.
\end{definition}

Let $G$ be a fat graph. We define the intersection graph $\Gamma(G)$ as follows: there is a vertex corresponding to each standard cycle and there is an edge between two nodes if the corresponding standard cycles intersect.

In this paper, by a subgraph of a fat graph $\Gamma$ we mean the unions of some of the standard cycles of $\Gamma.$


\section{Minimal non-admissible fat graph}

In this section, we study non-admissible fat graph. We give a constructive proof of Theorem \ref{thm:1.2}. Recall that, for a given fat graph $G$ and a positive real number $l$ if there exists a metric on $G$ such that the length of each standard cycle is equal to $l$ and the length of each non-standard cycle is strictly greater than $l$, we say that the graph is combinatorially admissible. A fat graph is non-admissible if it is not admissible. Also, a non-admissible fat graph is called minimal non-admissible if every subgraph of $G$ is admissible. We prove the following:
\begin{theorem}\label{inf-non-admissible}
There are infinitely many minimal non-admissible fat graphs.
\end{theorem}

Suppose $G$ is a given $4$-regular fat graph such that the intersection graph $\Gamma(G)$ is a planar graph. Consider a standard cycle $C=e_1*\cdots *e_n$ and let $v_1, \dots , v_n$ be the nodes of $C$ enumerated in a fixed orientation. Let $C_i$ denote the standard cycle meeting $C$ at $v_i$. If the orientation induced from the plane gives the cyclic ordering $C_1<C_2<\cdots < C_n$ (clockwise or anti-clockwise) to the set of nodes adjacent to the node $C$ in $\Gamma(G)$ then we say that the node $C$ in $\Gamma(G)$ respects an orientation of the fat graph. If each of the nodes of $\Gamma(G)$  respects an orientation of the fat graph then we say that the intersection graph respects an orientation of the fat graph.

Suppose $G$ is a fat graph such that the intersection graph $\Gamma(G)$ is planar and respects an orientation of $G$. We assume that $G$ is $4$-regular and two standard cycles can intersect at most once. Let the number of faces in $\Gamma(G)$ be $f$ and the number of nodes be $v$.

\begin{lemma}\label{vf-obstruction}
Suppose $v\leq f$, then $G$ is non-admissible.
\end{lemma}
\begin{proof}
Each node of $\Gamma(G)$ corresponds to a standard cycle of $G$. Also, there is a natural association of a non-standard cycle to each face of $\Gamma(G)$ so that the non-standard cycles associated to distinct faces have no edges in common. Note that, this correspondence is not uniquely defined when a standard cycle is intersected by exactly two other standard cycles. But this non-uniqueness is not a problem.

Assume that we are given a metric on the graph $G$ (i.e., lengths associated to each edge) so that all standard cycles have the same length $\lambda$. Then as each edge is in a unique standard cycle, the total length of the edges is $v\lambda$. On the other hand, each edge is in the non-standard cycle corresponding to a unique face of $\Gamma(G)$. Hence if $\mu$ is the \emph{average} length of the non-standard cycles corresponding to the faces, then $f\mu = v\lambda$

Thus, if $v\leq f$ and $\mu \leq \lambda$ then some non-standard cycle has length at most $\lambda$. Thus the metric is not admissible.
\end{proof}

\begin{proof}[Proof of Theorem~\ref{inf-non-admissible}]
We first sketch the idea of the proof.

\noindent To prove the theorem, for each integer $n\geq 3$ we construct a fat graph $G_n$ with $n+1$ standard cycles and show that the graph  $G_n$ is minimal non-admissible. This is done in three steps. In the first step, for each $n\geq 3$ we construct the fat graph $G_n$. In the next step, we show that the graph is non-admissible. Finally,  we prove the minimality.

\subsection*{Construction of  $G_n$:}
The standard cycles $C_i$,  $i=0, 1, \ldots, n$, are described below (see Figure~\ref{sec2fig3}): $C_0=(v_{0,1}, v_{0,2}, v_{0,3}, \ldots , v_{0,n})$  and $C_i=(v_{i,0}, v_{i,1}, v_{i,2}),$ $i=1,\dots, n.$ We identify the nodes on the cycles by following $v_{0,i}=v_{i,0}\ \textrm{and}\ v_{1,2}=v_{2,2}, v_{2,1}=v_{3,1}, \ldots , v_{n,1}=v_{1,1}.$

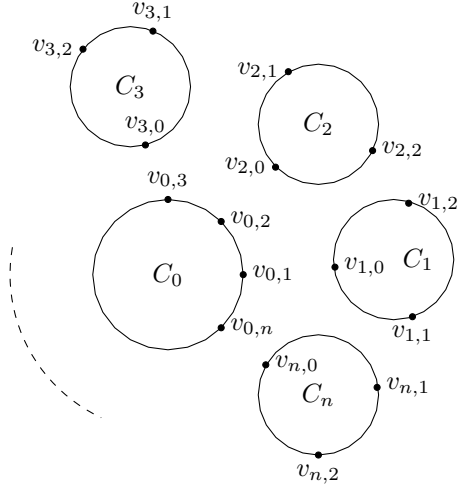
\begin{figure}[htbp]
\begin{center}
\begin{tikzpicture}
\draw  [domain=0:360] plot ({cos(\x)}, {sin(\x)});  \draw [fill] (1, 0) node [right] {$v_{0,1}$} circle [radius=0.04]; \draw [fill] (0.7071, 0.7071) node [right] {$v_{0,2}$} circle [radius=0.04]; \draw [fill] (0, 1) node [above] {$v_{0,3}$} circle [radius=0.04]; \draw [dashed] [domain=170:245] plot ({2.1*cos(\x)}, {2.1*sin(\x)}); \draw [fill] (0.7071, -0.7071) node [right] {$v_{0,n}$} circle [radius=0.04]; \draw (0,0) node {$C_0$};

\draw [domain=0:360] plot ({2+0.8*cos(\x)}, {-1.6+0.8*sin(\x)}); \draw [fill] (1.3, -1.2) node [right] {$v_{n,0}$} circle [radius=0.04]; \draw[fill] (2.78, -1.5) node [right] {$v_{n,1}$} circle [radius=0.04]; \draw[fill] (2, -2.4) node [below] {$v_{n,2}$} circle [radius=0.04]; \draw (2,-1.6) node {$C_n$};

\draw [domain=0:360] plot ({2+0.8*cos(\x)}, {2+0.8*sin(\x)}); \draw (2,2) node {$C_2$}; \draw [fill] (1.43, 1.43) node [left] {$v_{2,0}$} circle [radius=0.04];\draw [fill] (1.6, 2.7) node [left] {$v_{2,1}$} circle [radius=0.04]; \draw [fill] (2.72, 1.65) node [right] {$v_{2,2}$} circle [radius=0.04];

\draw [domain=0:360] plot ({-0.5+0.8*cos(\x)}, {2.5+0.8*sin(\x)}); \draw (-0.5,2.5) node {$C_3$}; \draw [fill] (-0.3, 1.725) node [above] {$v_{3,0}$} circle [radius=0.04]; \draw [fill] (-0.2, 3.24) node [above] {$v_{3,1}$} circle [radius=0.04]; \draw [fill] (-1.13, 3) node [left] {$v_{3,2}$} circle [radius=0.04];

\draw [domain=0:360] plot ({3+ 0.8*cos(\x)}, {0.2+0.8*sin(\x)}); \draw (3, 0.2) node [right] {$C_1$}; \draw [fill] (2.22, 0.1) node [right] {$v_{1,0}$} circle [radius=0.04]; \draw [fill] (3.2, 0.95) node [right] {$v_{1,2}$} circle [radius=0.04]; \draw [fill] (3.25, -0.56) node [below] {$v_{1,1}$} circle [radius=0.04];
\end{tikzpicture}
\end{center}
\caption{The schematics for building the graph
$G_n$}\label{sec2fig3}
\end{figure}
\subsection*{Non-admissibility of $G_n$:}
The intersection graph is $\Gamma(G_n)=(V,E)$ where $V=\{ v_i| i=0, 1, 2 ,\ldots, n\}$, $v_i$ corresponds to the standard cycle $C_i$ and $E=\{e_i, f_i|i=1, 2,\ldots, n\}$ where $e_i$ is the simple edge between $v_i$ and $v_{i+1}$ for $1\leq i \leq n-1$, $e_n$ is the edge between $v_n$ and $v_1$, $f_i$ is the edge between $v_0$ and $v_i$. The graph $\Gamma(G_n)$ is a prism over a $n$-sided polygon and respects an orientation of $G$. In a planar representation of $\Gamma(G_n)$ the number of nodes is the same as the number of faces. Hence by Lemma~\ref{vf-obstruction} the fat graph $G$ is non-admissible.

\subsection*{Minimality of $G_n$:}
Now, we show that, if we delete any cycle $C_i$ from $G_n$ then the resulting graph becomes admissible. Let us denote the fat graph obtain by removing the cycle $C_i$ from $G_n$ by $G_n^i$. Note that $G_n^i$ and $G_n^j$ are isomorphic for all $i,j\geq 1$. Therefore, it is enough to show that $G_n^0$ and $G_n^1$ are admissible.

\noindent In $G_n^0$ every standard cycle consists of two edges and every non-standard cycle consists of at least three edges. Hence, we define the metric $l : E \longrightarrow \mathbb{R}_+$ by $l(e)=\frac{1}{2}$ for each edge $e$ in $E$ follows that the length of each standard cycle is 1 and the length of each non-standard cycle is at least $\frac{3}{2}.$

\noindent The standard cycles of $G_n^1$ are given by, $C_0 = (v_{0,2}, v_{0,3}, \ldots , v_{0, n-1})$, $ C_2= (v_{2,0}, v_{2,1})$, $C_i=(v_{i,0}, v_{i,1}, v_{i, 2})\;\text{where}\; 2 < i < n\;\text{and}\;C_n= (v_{n,0}, v_{n,1}).$ The nodes are identified by the following relations: $v_{0,i}=v_{i,0}$, $v_{2,1}=v_{3,1}$, $v_{3,2}=v_{4,2}$, and $v_{4,2}=v_{5,2}, \ldots , v_{n-1,1}=v_{n,1}$.
\begin{figure}[htbp]
\begin{center}
\begin{tikzpicture}
\draw  [domain=0:360] plot ({cos(\x)}, {sin(\x)});  \draw [fill] (0.7071, 0.7071) node [right] {$v_{0,2}$} circle [radius=0.04]; \draw [fill] (0, 1) node [above] {$v_{0,3}$} circle [radius=0.04];  \draw [fill] (0, -1)node [below] {$v_{0,n-1}$} circle [radius=0.04];\draw [dashed] [domain=170:245] plot ({2.1*cos(\x)}, {2.1*sin(\x)}); \draw [fill](0.7071, -0.7071) node [right] {$v_{0,n}$} circle [radius=0.04]; \draw (0,0) node {$C_0$};\draw (2.2,-2.5) node {$C_n$}; \draw (2,2) node {$C_2$};\draw (-0.5,2.5) node {$C_3$}; \draw [domain=0:360] plot ({2.2+0.8*cos(\x)}, {-1.2+0.8*sin(\x)}); \draw [domain=0:360] plot ({2+0.8*cos(\x)}, {2+0.8*sin(\x)}); \draw [domain=0:360] plot ({-0.5+0.8*cos(\x)}, {2.5+0.8*sin(\x)}); \draw [fill] (-0.3, 1.725) node [above] {$v_{3,0}$} circle [radius=0.04]; \draw [fill] (-0.2, 3.24) node [above] {$v_{3,1}$} circle [radius=0.04]; \draw [fill] (-1.13, 3) node [left] {$v_{3,2}$} circle [radius=0.04];\draw [fill] (1.43, 1.43) node [left] {$v_{2,0}$} circle [radius=0.04];\draw [fill] (2.6, 2.52) node [right] {$v_{2,1}$} circle [radius=0.04]; \draw [fill] (1.42, -1) node [right] {$v_{n,0}$} circle [radius=0.04]; \draw[fill] (2.94, -1.5) node [right] {$v_{n,1}$} circle [radius=0.04]; \draw [domain=0:360] plot ({0.8*cos(\x)}, {-3+0.8*sin(\x)});\draw [fill] (0, -2.2) node [above] {$v_{n-1,0}$} circle [radius=0.04]; \draw [fill] (0.6, -3.5) node [right] {$v_{n-1,1}$} circle [radius=0.04];\draw [fill] (-0.6, -3.5) node [left] {$v_{n-1,2}$} circle [radius=0.04];
\end{tikzpicture}
\end{center}
\caption{ The fat graph $G_n^1.$} \label{sec2fig5}
\end{figure}
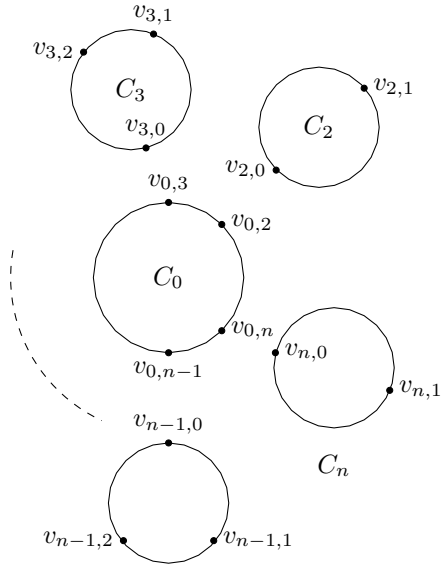
 We define $d: E(G_n^1) \longrightarrow \mathbb{R_+}$ as follows,
\begin{eqnarray*}
d(v_{0,i}, v_{0,i+1})&:=& \frac{1}{n-1}; 2\leq i\leq n-1,\\d(v_{0,n}, v_{0,2})&:=&\frac{1}{n-1}, \\\\ d(v_{j,1}, v_{j,2})&:=&\frac{1}{n-1}-\epsilon\\ d(v_{j,0},v_{j,1})= d(v_{j,0},v_{j,2})&:=& \frac{1}{2}+\frac{\epsilon}{2}-\frac{1}{2(n-1)};  3 \leq j \leq n-1\ \ \textit{and}\\ d(v_{2,0}, v_{2,1})=d(v_{n,0}, v_{n,1})&:=&\frac{1}{2}.
\end{eqnarray*}
We choose $\epsilon \in\mathbb{R}_+$ so that $d$ is a positive function and an admissible metric, namely one can choose a positive $\epsilon$ strictly less than $\frac{1}{n-1}$. This gives a metric with every standard cycle having length $1$. Note that the length of the outer non-standard cycle is $ 2 - (n-2)\epsilon$ which is strictly greater than one. If $C$ is any non-standard cycle other than the outer cycle then it must consist of at least three edges where at least two of them have length at least $\frac{1}{2}+\frac{\epsilon}{2}-\frac{1}{2(n-1)} $ and at least one edge of length $\frac{1}{n-1}$. Thus we have,$$ \text{length}(C) \geq 2\left(\frac{1}{2}+\frac{\epsilon}{2}-\frac{1}{2(n-1)}\right)+ \frac{1}{n-1}=1+\epsilon > 1.$$
\end{proof}

\subsection{More minimal non-admissible graphs}
In this situation we have the following natural question.

\begin{question}
Does $\{G_n|n\in \mathbb{N}, n \geq 3\}$ exhaust the set of all minimal non-admissible fat graphs?
\end{question}

\noindent We see that this is not the case. Consider the following example.

\noindent  Consider the fat graph
\begin{equation}
\begin{aligned}
G = {} & \{ v_1:[v_2, v_4, v_3, v_8], v_2:[v_1, v_4, v_3, v_6], v_3:[v_1, v_7, v_2, v_6],  \\ & v_4:[v_1, v_2, v_8, v_5], v_5:[v_4, v_8, v_6, v_7], v_6:[v_2, v_3,v_5, v_7], \\ &v_7:[v_3, v_5, v_6, v_8],  v_8:[v_1, v_7, v_4, v_5] \}
\end{aligned}
\end{equation}
where, $v_i, i=1,2, \ldots, 8$ are the nodes of $G$ and $v_i:[u_1, \ldots , u_{m_i}]$ means $u_i$, $i=1,2, \ldots m_i$ are nodes adjacent to $v_i$ with the order $(v_i, u_1) < (v_i, u_2) < \cdots < (v_i, u_{m_i})$ on the set of edges incident at $v_i$. Here $(u, v)$ denotes a simple edge between the nodes $u$ and $v$. The intersection graph of $G$ is planar and respects an orientation of $G$. In a planar representation of the intersection graph the number of nodes is $5$ which is equal to the number of faces. It follows by Lemma~\ref{vf-obstruction} that $G$ is non-admissible. Now, we claim that $G$ is minimal.

\noindent We show that if we delete any standard cycle $c$ from $G$, then the resulting graph $G-c$ is admissible. First, see that $G-c_2$ where, $c_2=(v_2, v_4)*(v_4, v_5)*(v_5, v_6)*(v_6, v_2)$ is admissible. The intersection graph of $G-c_2$ is a rectangle. Hence, each standard cycle consists of exactly two edges. On the other hand each non-standard cycle consists of at least four edges. So we assign length $\frac{1}{2}$ to each edge which makes $G-c_2$ admissible.

Now we define a metric on $G-c_3$ where, $c_3=(v_3, v_6)*(v_6,v_7)*(v_7, v_3)$. To find a metric on the fat graph we need to solve the following system of equations and inequations:
\begin{itemize}
\item $\sum_{e \in c} l(e)=1$ for each standard cycle $c$,
\item $\sum_{e \in d} l(e)>1$ for each non-standard cycle $d$, and
\item $l(e)>0$ for each edge $e$ in the graph $G$.
\end{itemize}

We use the Z3 SMT solver to solve the system. After renaming the nodes, the graph $G-c_3$ is given by,
\begin{equation}
  \begin{aligned}
    G-c_3 = {} & \{v_1:[v_2, v_4, v_2, v_5], v_2:[v_1, v_3, v_1, v_4], v_3:[v_2, v_5, v_4, v_5],\\
    & v_4:[v_1, v_3, v_5, v_2], v_5:[v_1, v_3, v_4, v_3]\}.
 \end{aligned}
\end{equation}

We define, $l(v_1, v_4)=x_0$, $l(v_3, v_4)= x_1$, $l(v_4, v_5) = x_2$, $l(v_2, v_4)=x_3$, $l(v_1, v_5)= x_4$, $l(v_3, v_5)=x_5$, $l(v_1, v_2)=x_6$ and $l(v_2, v_3)= x_7.$ Each solution of the system will provide an admissible metric on the fat graph. Using the Z3 SMT solver we have that the above system is satisfiable. i.e., the graph is admissible and a solution is given by $ x_0 =\frac{1}{2}$, $x_1 = \frac{1}{2}$, $x_2 =\frac{1}{8}$, $x_3 = \frac{1}{8}$,  $x_4 = \frac{3}{8}$, $x_5 = \frac{1}{2}$,  $x_6 = \frac{1}{2}$ and $x_7 = \frac{3}{8}$. By symmetry, we have $G-c_1=G-c_3$ where $c_1=(v_1, v_2)*(v_2, v_3)*(v_3,  v_1)$ and thus is admissible.

Also, $G-c_4=G-c_5$, where $c_4=(v_1, v_8)*(v_8, v_4)*(v_4, v_1)$ and $c_5=(v_5, v_8)*(v_8, v_7)*(v_7,  v_5)$. Hence, it suffices to show that $G-c_4$ is  admissible. After renaming the nodes the graph $G-c_4$ is given by,
 \begin{equation}
   \begin{aligned}
     G-c_4 = {} & \{v_1:[v_2, v_5, v_4, v_5], v_2:[v_1, v_3, v_4, v_3], v_3:[v_2, v_5, v_2, v_4], \\
     & v_4:[v_1, v_5, v_2, v_3],v_5:[v_1, v_4, v_1, v_3]\}.
  \end{aligned}
\end{equation}

As before we define, $l(v_1, v_4)= x_0$, $l(v_4, v_5)= x_1$, $l(v_2, v_4)= x_2$, $l(v_3, v_4)= x_3$, $l(v_1, v_5)= x_4$, $l(v_3, v_5)= x_5$, $l(v_1, v_2)= x_6$ and $ l(v_2, v_3)= x_7.$ Using the Z3 SMT solver we see that the fat graph is admissible and a solution is given by $x_0 = \frac{3}{4}$, $ x_1 =\frac{3}{8}$, $ x_2 = \frac{1}{8}$, $x_3 =\frac{1}{2} $, $x_4 = \frac{1}{2}$, $ x_5 = \frac{1}{8}$, $x_6 = \frac{1}{8}$ and $x_7 =\frac{1}{2}$.


\section{Uni-trivalent graph of large girth}
A graph is called trivalent if the degree of each node of the graph is 3. The length of a shortest cycle of a graph $G$ is called the girth of the graph and is denoted by $T(G)$. Let $f(n)$ be the smallest number for which there exists a trivalent graph of girth at least $n$ with $f(n)$ nodes. It follows from the result in \cite{PH}(1963) that $f(n)$ satisfies $$3\cdot 2^{\left[(n-3)/2\right]+1}-2\leq f(n)\leq 2^{n+1}-1.$$ In~\cite{PB}(1978), it has been shown that, for all $n\in \mathbb{N}$ and $g\geq g_{n}$ there exists a trivalent graph $G$ with $|G|=2g-2$ nodes and girth $T(G)\geq n$, where
\[ g_n = \left\{
  \begin{array}{l l}
    2 & \quad \text{if $\ n=1,2$, }\\
    n+1 & \quad \text{if $\ n=3, 4, 5$, }\\
    2^n & \quad \text{if $\ n \geq 6.$}
  \end{array} \right.\]
In~\cite{NB}(1998), the author gave a constructive prove to show that $f(n)\leq 2^n$.
Recall that, a graph $G=(V, E)$ is called a uni-trivalent graph if there is a node $v_0\in V$ of degree one and all other nodes have degree three. In this section, we prove the following lemma.

\begin{lemma}\label{sec3lem1}
Given any $n\in \mathbb{N}$, there exists a uni-trivalent graph $G$ with $f(n)+2$ nodes and girth $T(G)\geq n.$
\end{lemma}

\begin{proof}
First we consider a trivalent graph of girth $\geq n$ with $f(n)$ nodes where the existence follows from~\cite{PH},~\cite{PB},~\cite{NB}. Next, we delete one edge and then introduce two more nodes and three edges to obtain a uni-trivalent graph.

\noindent Consider the trivalent graph $G'=(V', E')$ of girth at least $n$ with $f(n)$ nodes. Then fix an edge $e=(x, y)$ in $G'$. Now, the uni-trivalent graph of girth $\geq n$ is given by $G=(V, E)$ where $V = V' \cup \{u, v\}$, $ E= (E'\setminus \{e\}) \cup \{(u, x), (u, y), (u,v)\}.$

\end{proof}

\section{Hyperbolic pair of pants}\label{pants}
In this section, we prove three lemmas on hyperbolic pair of pants which will be needed for subsequent sections. By a hyperbolic pair of pants we mean a pair of pants equipped with a hyperbolic structure. Unless otherwise noted, we assume that pairs of pants are hyperbolic. It is a fact, in hyperbolic geometry that on a pair of pants the lengths of its boundary components  determine a unique hyperbolic structure up to isometry where the boundary curves are geodesic~\cite{CB}. For any triple $(l_1, l_2, l_3)$ of positive real numbers, $P(l_1, l_2, l_3)$ denotes the hyperbolic pair of pants with geodesic boundaries $\gamma_1, \gamma_2$ and $\gamma_3$ of lengths $l_1, l_2$ and $l_3$ respectively.
\begin{definition}
Let $P$ be a hyperbolic pair of pants with boundary components $\gamma_1, \gamma_2$ and $\gamma_3$. The length of the simple geodesic arc $\delta_i$ with both end points at $\gamma_i$ meeting perpendicularly is called the \emph{height} of the hyperbolic pair of pants $P$ with respect to the waist $\gamma_i$.
\end{definition}

\begin{lemma}\label{pants:lem1}
Given any positive number $l$, there exists a pair of pants $P=P(l,kl,kl)$ for some $k$ such that $length(\delta_1) \geq l.$
\end{lemma}
\begin{proof}
Let us fix a $k\in \mathbb{R}_{>0}$ and consider the hyperbolic pair of pants $P=P(l,kl,kl)$. We consider the height $\delta_1$ which has its endpoints on the boundary curve $\gamma_1$ of length $l$. For symmetry reasons, the endpoints of $\delta_1$ divide $\gamma_1$ into two equal segments of length $\frac{l}{2}$. The union of $\delta_1$ with one of these segments forms a piecewise geodesic curve of length $\frac{l}{2}+l(\delta_1)$. But this curve is freely homotopic to one of the boundary components of length $kl$. Therefore, we have $\frac{l}{2}+l(\delta_1)> kl\Rightarrow l(\delta_1)> kl-\frac{l}{2}$ and hence we conclude that the lemma is true for all $k\geq \frac{3}{2}$
\end{proof}

\begin{lemma}\label{pants:lem2}
The height $\delta$ of the hyperbolic pair of pants $P(l,l,l)$ satisfies $$length(\delta)> \frac{l}{2}.$$
\end{lemma}
\begin{proof}
The proof of this lemma is a special case of the proof of Lemma~\ref{pants:lem1}, inparticular when $k=1$.
\end{proof}
\begin{lemma}\label{pants:lem3}
The distance between any two distinct boundary components of $P(l,l,l)$ is given by $$ dist(\gamma_i, \gamma_j)=arc\sinh\left(\frac{1}{2\sinh\frac{l}{4}}\right) \hspace{10pt} \text{for} \hspace{10pt} i\neq j.$$
\end{lemma}
\noindent Lemma~\ref{pants:lem3} follows from simple hyperbolic trigonometry.

\section{Polygonal Quasi-geodesic}
In this section, we develop three lemmas (Lemmas \ref{quasigeod}, \ref{corridor} and \ref{ratio}) which are used in the next section to prove the first main theorem. The first lemma says that a piecewise geodesic path with interior angles at the vertices bounded below and the lengths of the edges large enough (in terms of the lower bounds on the angles) is a quasi-geodesic. The second lemma says that if the length of a corridor is sufficiently large, the ratio of the length of a geodesic segment in the corridor with end points in the geodesic sides to the length of the corridor is close to $1$.

It is a fact in hyperbolic geometry that there exists a unique simple closed geodesic of minimal length in the free homotopy class of  a simple closed curve in a hyperbolic surface. The third lemma relates the length of a piecewise geodesic simple closed curve to the length of the simple closed geodesic in its free homotopy class.

\subsection{Quasi-geodesic}
Let $(X,d)$ be a metric space. Let $I\subset \R$ be a (possibly unbounded) interval. For $\lambda \geq 1$ and $\epsilon \geq 0$,  a $(\lambda, \epsilon)$-quasi-isometric embedding is a map $f:I\longrightarrow X$ satisfying $$\frac{1}{\lambda}|a-b|-\epsilon \leq d(f(a), f(b)) \leq \lambda |a-b|+\epsilon,$$ for all $a,b \in I$. If the restriction of $f$ to any subsegment $[x,y] \subset I$ of length at most $L$ is a $(\lambda, \epsilon)$-quasi-isometric embedding, then we call $f$ a $(L, \lambda, \epsilon)$-local quasi-isometric embedding. Note that, a quasi-isometric embedding need not be continuous. Now, we recall the definition of quasi-geodesic.
\begin{definition}\label{sec4def1}
A curve $\gamma : I \longrightarrow X $ in a geodesic metric space $(X,d)$ is called a $(\lambda, \epsilon)$-quasi-geodesic for some $\lambda \geq 1$ and $\epsilon \geq 0$, if the following inequality
\begin{equation}\label{sec4eq1}
\frac{1}{\lambda} l\left(\gamma|_{[t_1,t_2]}\right) - \epsilon \leq d(\gamma (t_1), \gamma (t_2))
\end{equation}
holds for all $t_1, t_2 \in I$.
\end{definition}

\begin{remark}
In the Definition~\ref{sec4def1},
if $\epsilon=0$ then $\gamma$ is simply called a
$\lambda$-quasi-geodesic.
\end{remark}
\noindent If the restriction of $\gamma$ to any subsegment $[a,b]\subset I$ of length at most $L$ is a $(\lambda, \epsilon)$-quasi-geodesic then we call $\gamma$ is a $(L,\lambda,\epsilon)$-local quasi-geodesic.

\subsection{Technical lemmas}\label{sec4:technical_lemmas}
\begin{lemma}\label{quasigeod}
Let $\gamma:\mathbb{R}\to \mathbb{H}$ be a piecewise geodesic curve such that the interior angles are bounded below by some $\theta_0\in \mathbb{R}_+$. If the length of a smallest geodesic piece in $\gamma$ is sufficiently large then there exist $k\geq 1$ and $\epsilon\geq 0$ such that $\gamma$ is a $(k, \epsilon)$-quasi-geodesic.
\end{lemma}

\begin{proof}
We first prove a result concerning the union of two geodesic segments and conclude that $\gamma$ is a locally quasi-geodesic. We then conclude that it is a quasi-geodesic.

\begin{lemma}\label{2segments}
Suppose $\beta_i:[0, l_i] \to \mathbb{H}$, $i=1, 2$ are geodesic segments parameterized by arc length with $\beta_1(l_1)=\beta_2(0)$ and the interior angle at the connecting point is $\alpha$. Then the piecewise geodesic arc $\beta=\beta_1*\beta_2: [0, l_1+l_2]\to \mathbb{H}$ defined by
\[ \beta(t) = \left\{
  \begin{array}{l l}
    \beta_1(t) & \quad \text{if $t \in [0,l_1]$ }\\
    \beta_1(t-l_1) & \quad \text{if $t \in [l_1, l_1+l_2]$ }
  \end{array} \right.\]
is a $k(\alpha)$-quasi-geodesic for some constant $k(\alpha)\geq 1.$
\end{lemma}

\begin{proof} Consider $\Delta(ABC)$ is a triangle in the Euclidean plane with vertices at $A, B$ and $C$. Suppose  $d(A, C)=t_1$, $ d(A,B) = t_2$ and the interior angle at $A$ is $\alpha.$
\subsection*{Claim:} There exists a real number $k(\alpha)\geq 1$ which depends only on $\alpha$, such that
\begin{equation}\label{sec4eq3}
\frac{1}{k(\alpha)}\left(t_1+t_2\right)\leq d(B,C).
\end{equation}
Suppose the claim is true. We prove that $\beta$ is a $k(\alpha)$-quasi-geodesic. Let $P=\beta(t_1), Q=\beta(t_2), \; t_1\leq t_2$ be two points on $\beta$. If $t_1, t_2$ satisfy  $0\leq t_1, t_2 \leq l_1$ or $l_1\leq t_1, t_2 \leq l_1 + l_2$ then we have $l_{\mathbb{H}}\left(\beta|_{[t_1, t_2]}\right) = d_{\mathbb{H}}(P, Q).$ Hence, for any $k(\alpha)\geq 1$ we have $$\frac{1}{k(\alpha)}l_{\mathbb{H}}(\beta|_{[t_1, t_2]})\leq d_\mathbb{H}(P, Q).$$  Now let $0\leq t_1 \leq l_1$ and $l_1 \leq t_2 \leq l_1 +l_2$. Consider the Euclidean triangle $\Delta(ABC)$ such that $d(A,C)=t_1, d(A,B)=t_2$ and the interior angle at $A$ is $\alpha$. Then we have $\frac{1}{k(\alpha)}(t_1 + t_2) \leq d(B,C)$ which follows from the claim. Also, by Toponogov's comparison theorem (see~\cite{CD}), we have $d(B,C)\leq d_{\mathbb{H}}(P,Q).$ Hence, we have $\frac{1}{k(\alpha)}(t_1+t_2) \leq d_{\mathbb{H}}(P,Q)$ which is the same as the following inequality $$\frac{1}{k(\alpha)} l_{\mathbb{H}}(\beta|_{[t_1,t_2]}) \leq d_{\mathbb{H}}(P, Q).$$  Hence, $\beta$ is a $k(\alpha)$-quasi-geodesic. Thus it suffices to prove the claim.

\begin{proof}[Proof of the claim]
 Let $d(B, C)= t_3$. From the triangle $ABC$ we have
\begin{eqnarray*}
t_3^2&=&t_1^2+t_2^2-2t_1t_2\cos\alpha\\ \Rightarrow \left(\frac{t_3}{t_1+t_2}\right)^2 &=& 1-\frac{2t_1t_2}{(t_1+t_2)^2}(1+\cos\alpha)
\end{eqnarray*}
Now using the inequality of arithmatic and geometric means we have
\begin{eqnarray*}
\frac{t_1+t_2}{2} &\geq& \sqrt{t_1t_2} \\ \Rightarrow 1-\frac{2t_1t_2}{(t_1+t_2)^2}(1+\cos\alpha) &\geq & \frac{1-\cos\alpha}{2}\\ \Rightarrow \left(\frac{t_3}{t_1+t_2}\right)^2 &\geq & \frac{1-\cos\alpha}{2} \\ \Rightarrow \frac{t_3}{t_1+t_2} &\geq & \sin\left(\frac{\alpha}{2}\right).
\end{eqnarray*}
Therefore we have $\frac{1}{k(\alpha)}(t_1+t_2)\leq t_3$ where $k(\alpha)=\frac{1}{\sin\left(\frac{\alpha}{2}\right)}$.
\end{proof}

Thus, Lemma~\ref{2segments} follows.
\end{proof}

It follows by Lemma~\ref{2segments} that $\gamma$ is a $L_0$-local $k(\theta_0)$-quasi-geodesic where $L_0$ is the length of a smallest geodesic piece in $\gamma$.

Thus, by Theorem~4 in~\cite{JC}, it follows that for sufficiently large $L_0$ there exist $k\geq 1$ and $\epsilon > 0$ such that $\gamma$ is a $(k, \epsilon)$-quasi-geodesic. This concludes the proof of Lemma~\ref{quasigeod}.
\end{proof}

\begin{definition}
Let $L$ be a hyperbolic line and $I\subset L$. For a positive real number $W\in \mathbb{R_+}$, the $W$-corridor about $I$ along $L$ is the set
\begin{equation}\label{sec4eq7}
W(I, L) = \{z \in H| d_{\mathbb{H}}(z, L) \leq W, \rho_L(z)\in I \}
\end{equation}
where $\rho_L$ is the orthogonal projection of $\mathbb{H}$ onto $L$.
\end{definition}
 Let $\gamma$ be a finite geodesic segment of the complete geodesic $\gamma'$ and $\delta$ be any geodesic segment in the corridor (see~~\cite{JC}) $W(\gamma, \gamma')$ with end points on the geodesic sides of $W(\gamma, \gamma')$ and $\delta$ lies on the closure of one of the components  $W(\gamma, \gamma') \setminus \gamma$. The length $l_{\mathbb{H}}(\gamma)$ of $\gamma$ we call the length of the corridor $W(\gamma, \gamma')$.

\begin{figure}[htbp]
\begin{center}
\begin{tikzpicture}
\draw (0,0) -- (0, 3.3) node [above] {$\gamma'$}; \draw [<->] (-3,0) -- (3, 0); \draw (0,0) -- (0.75*2, 0.75*4); \draw (0,0) -- (-2*0.75, 0.75*4); \draw [domain=63.8:116.2] plot ({cos(\x)}, {sin(\x)}); \draw [domain=63.8:116.2] plot ({3*cos(\x)}, {3*sin(\x)}); \draw [domain=146:169] plot ({5.2+5*cos(\x)}, {5*sin(\x)}); \draw (0,2)node [left] {$\gamma$}; \draw (-0.15,1.1) node {$a$}; \draw (-0.1,3.15) node {$b$}; \draw (0.5,2.1) node {$\delta$}; \draw (0.7,0.9) node {$a'$}; \draw (1.6,2.8) node {$b'$};
\end{tikzpicture}
\end{center}
\caption{Corridor}\label{sec4fig2}
\end{figure}
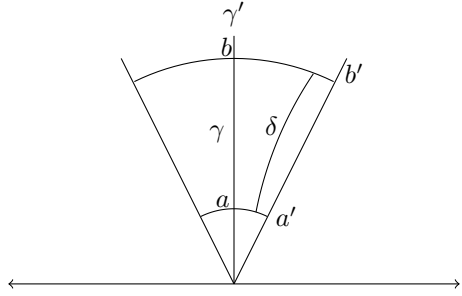
\begin{lemma}\label{corridor}
In the above setting, for a fixed $W>0$, if the length of the corridor is sufficiently large then the ratio $\frac{l_{\mathbb{H}}(\delta)}{l_{\mathbb{H}}(\gamma)}$ is arbitrarily close to $1$.
\end{lemma}

\begin{proof}
The fact $\rho_{\gamma'}(\delta)=\gamma $ where $\rho_{\gamma'}:\mathbb{H}\to \gamma'$ is orthogonal projection, implies that $1\leq \frac{l_{\mathbb{H}}(\delta)}{l_{\mathbb{H}}(\gamma)}$. Using triangle inequality, we have $\frac{l_{\mathbb{H}}(\delta)}{l_{\mathbb{H}}(\gamma)}\leq \frac{2W}{l_{\mathbb{H}}(\gamma)}+1$. Combining these two inequalities we have $$1\leq \frac{l_{\mathbb{H}}(\delta)}{l_{\mathbb{H}}(\gamma)}\leq \frac{2W}{l_{\mathbb{H}}(\gamma)}+1.$$ Hence $\frac{l_{\mathbb{H}}(\delta)}{l_{\mathbb{H}}(\gamma)}$ tends to $1$ as the length of the corridor tends to infinity.
\end{proof}

\begin{lemma}\label{ratio}
Let $\gamma$ be a piecewise geodesic, essential, simple closed curve of a hyperbolic surface $S$ with smallest interior angle $\theta_0>0$. Suppose $\tilde{\gamma}$ is a lift of $\gamma$ in the universal cover. If the length of the smallest geodesic piece of $\gamma$ is sufficiently large then:
\begin{enumerate}
\item There exist $k \geq 1$ and $\epsilon\geq 0$  such that $\tilde{\gamma}$ is a $(k , \epsilon)$-quasi-geodesic.

\item If $\gamma'$ is the simple closed geodesic in the free homotopy class of $\gamma$ then the ratio  $\frac{l_{\mathbb{H}}(\gamma)}{l_{\mathbb{H}}(\gamma')}$ tends to $1$ as the length of the smallest geodesic piece of $\gamma$ tends to $\infty$.
\end{enumerate}

\end{lemma}

\begin{proof}

Let $\theta_1, \theta_2, \ldots, \theta_n$ be the interior angles at the corners of $\gamma$. We write $\gamma=\gamma_1*\gamma_2*\cdots*\gamma_n$ where $\gamma_i$'s are geodesic segments in $\gamma$. We define
\begin{eqnarray*}
l_0&=&\text{min}\{l_{\mathbb{H}}(\gamma_i): i=1, 2,\ldots, n\}\hspace{5pt} \text{and}\\ \theta_0&=&\text{min}\{\theta_i: i=1, 2, \ldots, n\}.
\end{eqnarray*}
Now consider a lift $\tilde{\gamma}$ of $\gamma$ in the universal cover of $S$. It follows from Lemma~\ref{quasigeod} that $\tilde{\gamma}$ is a $l_0$-locally $k(\theta_0)$-quasi geodesic. Moreover, if $l_0$ is sufficiently large then there exist $k\geq 1$ and $\epsilon>0$ such that $\tilde{\gamma}$ is a globally $(k,\epsilon)$-quasi geodesic.

Let $A, B$ be the end points of $\tilde{\gamma}$ on the boundary at infinity. Let $\tilde{\gamma}'$ be the geodesic joining $A$ and $B$. Observe that $\tilde{\gamma'}$ projects onto $\gamma'$. It follows from Theorem 2 in~\cite{JC} that there exists $W>0$ such that  $\tilde{\gamma}$ is contained in $\text{Nbd}(\tilde{\gamma}', W)$, the $W$ neighbourhood of $\tilde{\gamma}'$.

Let $P = \tilde{\gamma_1} * \tilde{\gamma_2} * \cdots * \tilde{\gamma_n}$ be a continuous sub-segment of $\tilde{\gamma}$ such that $\tilde{\gamma}_i$'s project onto $\gamma_i$, $i=1, 2, \ldots, n$ and we have $l_{\mathbb{H}}(\gamma)=\sum\limits_{i=1}^n l_{\mathbb{H}}(\tilde{\gamma_i}).$
\begin{figure}[htbp]
\begin{center}
\begin{tikzpicture}
\draw (0, 0) circle [radius=2]; \draw [domain=229:311] plot ({cos(\x)}, {1.3+sin(\x)}); \draw [domain=229:270] plot ({1.3+cos(\x)}, {1.3+sin(\x)}); \draw [domain=272:311] plot ({-1.3+cos(\x)}, {1.3+sin(\x)}); \draw [domain=216:270] plot ({1.7+0.5*cos(\x)}, {0.6+0.5*sin(\x)}); \draw [domain=270:323] plot ({-1.66+0.5*cos(\x)}, {0.6+0.5*sin(\x)}); \draw [dotted] (2,0) -- (1.7, 0.1); \draw [dotted] (-2,0) -- (-1.7, 0.1); \draw (1, 0.17) node {$\tilde{\gamma}$}; \draw (-2, 0) node [left] {$A$}; \draw (2, 0) node [right] {$B$};

\draw (-2,0)--(2, 0); \draw (-1,0) node [below] {$\tilde{\gamma}'$}; \draw [domain=225:315] plot ({2.8284*cos(\x)}, {2+2.8284*sin(\x)}); \draw [domain=45:135] plot ({2.8284*cos(\x)}, {-2+2.8284*sin(\x)});

\draw [domain=162:198] plot ({3+2.41*cos(\x)}, {2.41*sin(\x)}); \draw [domain=180+162:180+198] plot ({-3+2.41*cos(\x)}, {2.41*sin(\x)}); \draw (0,0.1) node [below] {$\tilde{\gamma}_i'$}; \draw (0,0.5) node {$\tilde{\gamma}_i$};
\end{tikzpicture}
\end{center}
\caption{A lift of $\gamma$}\label{sec4fig3}
\end{figure}
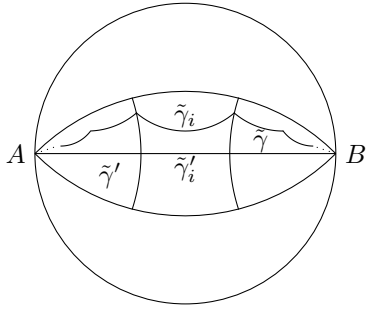
Suppose $P_0, P_1, \ldots, P_n$ are the points on the path $P$ such that $\tilde{\gamma}_i$ are the geodesic segments joining $P_{i-1}$ and $P_i$, $i=1, 2, \ldots ,n$. Denote the orthogonal projection of the point $P_i$ is by $P_i'$, $i=1,2, \ldots ,n$ and the geodesic segment joining $P_{i-1}'$ and $P_i'$  by $\tilde{\gamma_i}'$. Then we have $l_{\mathbb{H}}(\gamma') = \sum\limits_{i=1}^nl_{\mathbb{H}}(\tilde{\gamma_i}').$ The geodesic segment $\tilde{\gamma_i}$ lies in the $W$-corridor of $\tilde{\gamma}_i'$ along $[A,B]_{\mathbb{H}}$ such the end points are on the geodesic sides of the corridor. Then it follows from Lemma~\ref{corridor} that the ratio $\frac{l_{\mathbb{H}}(\tilde{\gamma_i})}{l_{\mathbb{H}}(\tilde{\gamma}_i')}$ tends to 1 if $l_0$ is sufficiently large, $i=1,2, \ldots ,n.$ Hence, $$\frac{l_{\mathbb{H}}(\gamma)}{l_{\mathbb{H}}(\gamma')} = \frac{\sum\limits_{i=1}^n l_{\mathbb{H}}(\tilde{\gamma_i})}{\sum\limits_{i=1}^n l_{\mathbb{H}}(\tilde{\gamma}_i')}$$ tends to 1 when $l_0$ is sufficiently large. This completes the proof.
\end{proof}

\section{Realization using hyperbolic surfaces with boundary}\label{S:boundary}

In this section, we prove that any  combinatorially admissible fat graph is realized as the systolic graph of some hyperbolic surface with totally geodesic boundary.

Namely, we construct hyperbolic cylinder corresponding to each standard cycle. Next, we plumb the hyperbolic cylinders together according to the intersections of the standard cycles to obtain a hyperbolic surface with boundary. We obtain a hyperbolic surface with totally geodesic boundary by cutting the surface along the geodesics in the free homotopy classes of the boundary components. We show that the systolic graph of the surface is isomorphic to the fat graph.

Let $G$ be a given decorated fat graph which is combinatorially admissible. The graph $G$ is the union of standard cycles with disjoint edges. As combinatorial admissibility is scale invariant, for each $l>0$ there is a metric $d_l$ on $G$ so that each standard cycle has length $l$ and each non-standard cycle has length greater than $l$. We can assume that the metrics $d_l$ coincide up to scaling.

Note that as each non-standard cycle has length bounded below by the length of a simple non-standard cycle and there are only finitely many simple non-standard cycles, there is a constant $r >1$ so that the length of any non-standard cycle in the metric $d_l$ is at least $r l$.

\subsection{Plumbing}
 Suppose $C_i,i=1,\dots, k$  are all the standard cycles of $G$. First, we construct hyperbolic cylinders  $C_i(l,\epsilon)$ which are the copies of $C(l,\epsilon)$ corresponding to the standard cycles $C_i$, where $l, \epsilon$ are positive real numbers. Note that, $C(l, \epsilon)$ denotes the hyperbolic cylinder (see~\cite{PB1}) of width $2\epsilon$ and the central geodesic of length $l$.

Next, we plumb the cylinders according to the intersection of the standard cycles in the fat graph $G$ in such a way that the central geodesics intersect transversally (inparticular at an angle $\frac{\pi}{2}$ for a node of valence $4$) and the path metric restricted to the union of the central geodesics is $d_l$. We fix the angles between the central geodesics depending on the valence. The obtained hyperbolic surface is denoted by $\Sigma_l(G)$. 

We remark that the boundary components of $\Sigma_l(G)$ are not geodesics, and not even piecewise geodesics. It is easy to see that the union of the central geodesics is a spine of $\Sigma_l(G)$ which is isomorphic to $G$ and each boundary component is freely homotopic to a non-standard cycle in the spine.

\subsection{Constructing surface with geodesic boundary}
Let $\gamma '$ be a boundary component of $\Sigma(G)$ which is freely homotopic to the simple closed curve $\gamma ''$ in the spine.
\begin{lemma}\label{sec5lem1}
For a suitable choice of the width $\epsilon$ and the length $l$ of the central geodesic of the cylinder $C(l,\epsilon)$, the following hold.
\begin{enumerate}
\item The unique geodesic $\gamma$ in the free homotopy class of $\gamma'$ lies in the (topological) cylinder with boundaries $\gamma''$ and $\gamma'$.

\item The length of $\gamma$ is strictly greater than $l$.
\end{enumerate}
\end{lemma}

\begin{proof}
Let the interior angles of $\gamma''$ be $\theta_1, \theta_2, \ldots ,\theta_k$ and $\theta=\min\{\theta_i|i=1,2\ldots, k\}$. Then $\theta(> 0)$ is a lower bound in the interior angles of the lift $\tilde{\gamma}''$ of $\gamma''$ in the universal cover of $\Sigma_l(G)$. It follows from Lemma~\ref{quasigeod} that $\tilde{\gamma}''$ is a $(k, \zeta)$-quasi geodesic for some $k\geq 1$ and $\zeta \geq 0$. Let us denote the axis of the quasi-geodesic $\tilde{\gamma}''$ by $\tilde{\gamma}$ which is the geodesic line in the hyperbolic plane joining the end points of $\tilde{\gamma}''$. By Theorem 2 in~\cite{JC} there is a positive number $W\in \mathbb{R}$ such that $\tilde{\gamma}''\subset Nbd(\tilde{\gamma}, W)$. Also, each interior angle is strictly less than $\pi$, hence $\tilde{\gamma}'$ and the geodesic $\tilde{\gamma}$ lie on the same side of $\tilde{\gamma}''$. Therefore, if we choose the positive number $\epsilon$ larger that $W$ then the geodesic $\tilde{\gamma}$ lies in the region bounded by $\tilde{\gamma}'$ and $\tilde{\gamma}''$.

By Lemma~\ref{ratio}, it follows that if $l$ is sufficiently large then $\frac{length_{d_l}(\gamma'')}{length_{d_l}(\gamma)}<r$. As $\gamma''$ is a non-standard cycle, $length_{d_l}(\gamma'') > rl$. Hence $length_{d_l}(\gamma) >l$ as claimed.
\end{proof}

We remark that the proof above shows that the length of any geodesic is at least $l$, as each geodesic is homotopic to a cycle in the spine and to a non-standard cycle unless it is a central geodesic. In the former case the above argument applies and in the latter case the length is $l$.

Thus, by choosing $l$ sufficiently large, we obtain a surface $\Sigma(G) = \Sigma_l(G)$ satisfying the following.

\begin{lemma}\label{bounded-surface}
There exists a positive real number $l$ such that $\Sigma(G)$ satisfies the following.
\begin{enumerate}
\item For each boundary component $\gamma'$ of $\Sigma(G)$, the simple closed geodesic $\gamma$ in the free homotopy class of $\gamma'$ lies in the (topological) cylinder bounded by $\gamma'$ and $\gamma''$ where the piecewise geodesic simple closed curve in the spine of $\Sigma(G)$ freely homotopic to $\gamma'$.

\item The length of each  closed geodesic in $\Sigma(G)$ which is not a central geodesic(in particular of the curves $\gamma$) is strictly greater than the length $l$ of the central geodesics.
\end{enumerate}
\end{lemma}

Hence, cutting off the surface $\Sigma(G)$ along these geodesics in the free homotopy classes of boundary components we get the
hyperbolic surface denoted by $\Sigma_1(G)$ with totally geodesic boundary such that the systolic graph of $\Sigma_1(G)$ is $G$.

\section{Capping}
In this section, we cap the hyperbolic surface $\Sigma_1(G)$ to obtain a closed hyperbolic surface which satisfies our desired conditions. The idea is the following. Suppose $\Sigma$ is a hyperbolic surface with boundary. We embed the surface $\Sigma$ isometrically into a closed hyperbolic surface $S$ such that they have the same systole and all geodesics realizing the systole are contained in $\Sigma$. Thus, the systolic graphs of the closed hyperbolic surface $S$ and $\Sigma$ are isomorphic.

\begin{theorem}\label{T:capping}
Let $\Sigma$ be a hyperbolic surface with totally geodesic boundary whose systoles are contained in the interior of $\Sigma$. There exists a closed hyperbolic surface $S$ and an isometric embedding $i : \Sigma \to S$ such that the following hold:
\begin{enumerate}
\item $Sys(S)=Sys(\Sigma)$.

\item The systolic graph of $S$ is isomorphic to the systolic graph of $\Sigma$.
\end{enumerate}
\end{theorem}

The surface $S$ is constructed using a uni-trivalent graph. Recall that, if $G$ is a trivalent graph and $l>0$ then we can associate to $G$ a hyperbolic surface by taking a pair of pants for each node with all three boundary components of length $l$ and identifying boundary components corresponding to edges of the graph. Namely, as each vertex $v$ is trivalent, we can identify the half-edges adjoining $v$ with the boundary components of the associated pair of pants. We identify boundary components corresponding to the two halfs of an edge. We call the image of a boundary component a \emph{rim}. Note that in general the gluing depends on  twist parameters (the  length and the twist are the Fenchel-Nielsen coordinates) but for our purposes these can be chosen arbitrarily.

We shall associate a surface with  boundary to a uni-trivalent graph $G$ in a similar way, except that we take pair of pants so that the length of the boundary component associated to the terminal edge (i.e., the edge containing the terminal vertex) is $l$ and all other boundary components have length $2l$. This gives a surface $\Sigma_l(G)$ with a single boundary component of length $l$. Note that all the rims have length $2l$.

We shall show that for an appropriate choice of graph  $G$, the surfaces $\Sigma_l(G)$ have properties that allow them to be used to cap $\Sigma$ to obtain the desired closed surface $S$. Define the continuous function $a(l), l\in \mathbb{R}_+$ by

\begin{equation}\label{sec6eq2}
a(l)= min \left\{2\operatorname{arcsinh}\left(\frac{1}{2\sinh\frac{l}{2}}\right), \operatorname{arccosh}\left(1+\frac{1+\cosh\frac{l}{2}}{\sinh^2l}\right)\right\}.
\end{equation}
Also, define the natural number $t= t(l), l\in \mathbb{R}_+$ by
\begin{equation}
t(l)= \left\lfloor \frac{l}{a(l)} \right\rfloor+1.
\end{equation}

\begin{lemma}\label{sec6lem1}
For given positive constant $l\in \mathbb{R}$, there exists a uni-trivalent graph $G$ of girth at least $t(l)$ such that
\begin{enumerate}

\item The length of any closed geodesic in $\Sigma_l(G)$ is greater than or equal to $l$.

\item If $\delta$ is an essential geodesic arc with endpoints on the geodesic boundary then the length of $\delta$ is greater than or equal to $l$.

\end{enumerate}
\end{lemma}

\begin{proof}
For given positive constant $l$, let $G$ be the uni-trivalent graph of girth $\geq t(l)$ where the existence is assured by Lemma~\ref{sec3lem1}. Now, consider the surface $\Sigma_l(G)$ with boundary. Observe that the rims of $\Sigma_l(G)$ are of length $2l$ and the boundary geodesic has length $l$. So, it is sufficient to show that no other geodesic $\sigma$ in $\Sigma_l(G)$ has length less than $l$. Without loss of generality we assume that $\sigma$ is a simple closed geodesic. Then $\sigma$ cannot be contained in a single pair of pants otherwise $\sigma$ would be a rim, which we have already excluded. Therefore, there is a partition $0=t_0< t_1< \cdots < t_n=1$ of $[0,1]$ and a sequence of  (not necessarily distinct) rims $\gamma_0, \gamma_1, \ldots , \gamma_n=\gamma_0$, so that $\sigma(t_i)\in \gamma_i, i=0,1,\ldots, n$ and each segment $\sigma_i=\sigma|_{[t_{i-1}, t_i]}$ lies in a single pair of pants, denoted by $Y_i$. Again a pair of pants $Y_i$ can occur more than once. Therefore, we have
\begin{equation}
l_\mathbb{H}(\sigma)= \sum\limits_{i=1}^n l_\mathbb{H}(\sigma_i).
\end{equation}

Now there are following three cases to be considered.

\subsection*{Case 1} Assume $\gamma_{i-1} \neq \gamma_i$ for all $i=1,2, \dots, n.$ It is easy to see that each rim corresponds to an edge of $G$. Hence $\sigma$ corresponds to a closed path in the graph $G$. We denote the closed path by $P(\sigma)$. The closed path $P(\sigma)$ contains a cycle (simple closed path) whose length is greater than or equal to the girth $T(G)$ of the graph $G$. Hence $n\geq t(l)$. For each $i$, $\sigma_i$ is a geodesic in $Y_i$ joining two distinct boundary components. If $Y_i$ is the pair of pants with boundary geodesics of length $2l$ then it follows from Lemma~\ref{pants:lem3} that,
\begin{equation}\label{in-eq:7}
l_\mathbb{H}(\sigma_i)\geq 2\operatorname{arcsinh}\left(\frac{1}{2\sinh\frac{l}{2}}\right).
\end{equation}
 If $Y_i=Y(l, 2l, 2l)$ then $\sigma_i$ is a geodesic with endpoints on the boundaries of length $2l$. In this case we have
\begin{equation}\label{in-eq:8}
l_\mathbb{H}(\sigma_i)\geq \operatorname{arccosh}\left(1+\frac{1+\cosh\frac{l}{2}}{\sinh^2l}\right).
\end{equation}
 Now from the definition of the function $a(l)$ in equation
\eqref{sec6eq2}, the inequations~\eqref{in-eq:7} and~\eqref{in-eq:8}, we have $l_\mathbb{H}(\sigma_i)\geq a(l).$ Therefore, \begin{eqnarray*}
l_\mathbb{H}(\sigma)=\sum\limits_{i=1}^n l_\mathbb{H}(\sigma_i)& \geq & n \cdot a(l) \geq t(l) \cdot a(l) \geq l.
\end{eqnarray*}

\subsection*{Case 2} Suppose there exists $i_0$ such that $\gamma_{i_0-1}=\gamma_{i_0}$ where $Y_{i_0}=Y(2l, 2l, 2l)$. Then it follows from the Lemma~\ref{pants:lem2} that $l_\mathbb{H} (\sigma_{i_0})\geq l$ which implies that $l_\mathbb{H}(\sigma) \geq l$.

\subsection*{Case 3} In the remaining case we assume that there is a $i_0$ such that $\gamma_{i_0-1} = \gamma_{i_0}$ where $Y_{i_0}=Y(l, 2l, 2l)$. Then both the end points of $\gamma_{i_0}$ are on the same boundary of length $2l$. In that case $l_\mathbb{H}(\sigma_{i_0})\geq \frac{l}{2}$. If there are two such different  $i_0$ and $j_0$ then we have two distinct geodesic segments $\sigma_{i_0}$ and $\sigma_{j_0}$, each of length greater than or equal to $\frac{l}{2}$. Therefore the length of $\sigma$ is greater than $l$. So we assume that there is only one such $i_0$. Hence the seams $\gamma_1,\ldots, \gamma_{i_0-2}, \gamma_{i_0+1}, \ldots, \gamma_n$ are pairwise distinct and correspond to a closed path $P$ in $G$. The length of the closed path $P$ is $(n-1)$ which is greater than or equal to $T(G)(\geq t(l))$. Therefore using arguments similar to Case~1 we conclude that $l_\mathbb{H}(\sigma)\geq l$.

Now we prove the second part of the lemma. Consider an essential geodesic arc $\delta$ in $\Sigma_l(G)$ with the end points on the boundary geodesic. If $\delta$ lies in the pair of  pants $Y_1=Y(l, 2l, 2l)$ then the length of $\delta$ is greater than or equal to the height of $Y(l, 2l, 2l)$ which is greater than  $l$ by Lemma~\ref{pants:lem2}. Hence we have $l_\mathbb{H}(\delta)\geq l$.

In the remaining cases, there is a sequence of rims $\gamma_0,
\gamma_1,\ldots,  \gamma_n=\gamma_0$ and a partition $0=t_0< t_1<\cdots< t_n=1$ such that $\delta(t_i)\in \gamma_i$ and no other rim is crossed  over. So, $\delta_i=\delta|_{[t_{i-1}, t_i]}$ lies in a single pair of pants, denoted by $Y_i$.  If $\gamma_i=\gamma_{i-1}$ for some $i$ then the length of $\delta_i$ is greater than or equal to $l$. So the length of $\delta$ is greater than or equal to $l.$

Now we assume that the segments $\gamma_i$ are distinct. The rims $\gamma_i$ determine a closed path $P$ in $G$ which contains a cycle. Hence, as in the proof of first part, we have $n\geq n_0$ and the length of each $\delta_i$ is greater than or equal to $a(l)$. Thus, the length of $\delta$ satisfies:
$$l_\mathbb{H}(\delta)=\sum\limits_{i=1}^n l_\mathbb{H}(\delta_i)\geq n\cdot a(l)\geq n_0\cdot a(l)\geq l.$$

\end{proof}

\subsection{Proof of the Theorem \ref{T:capping}}
\begin{proof}
Let $\Sigma$ be a hyperbolic surface with $b$ boundary components. Suppose the boundary components are $\gamma_i$ and $l_\mathbb{H}(\gamma_i)=l_i$ where $i=1,\dots, b$. We construct a closed hyperbolic surface as follows.

\subsection*{Step 1.} For each boundary component $\gamma_i$ of $\Sigma$, we consider the hyperbolic surface $\Sigma_{l_i}(G_i)$ with single boundary component as in Lemma~\ref{sec6lem1}. We denote the boundary component of $\Sigma_{l_i}(G_i)$ by $\delta_i$.
\subsection*{Step 2.}  For each $i=1,2, \ldots, b$, we glue the surface $\Sigma_{l_i}(G_i)$ with $\Sigma$ along the geodesic boundaries $\gamma_i$ and $\delta_i$ by an isometry and denote the obtained closed surface by $S$. Now, it remains to show that $S$ satisfies the conditions in Theorem \ref{T:capping}.

Suppose $\gamma$ is a shortest closed geodesic in $S$. If $\gamma$ is not contained in $\Sigma$ then $\gamma \cap \Sigma_{l_i}(G_i)\neq \phi$ for some $i\in\{1, \ldots, b\}$. Then $\Sigma_{l_i}(G_i)$ either contains $\gamma$ or contains an essential subarc of $\gamma$ with end points on the boundary of $\Sigma_{l_i}(G_i)$. In both cases it follows from Lemma~\ref{sec6lem1} that $l_{\mathbb{H}}(\gamma)\geq l_i > l$ which contradicts the assumption that $\gamma$ is a shortest closed geodesic in $S$.

Thus we conclude that if $\gamma$ is a shortest closed geodesic in $S$ then it is a shortest closed geodesic in $\Sigma$ as well, as claimed.
\end{proof}

Note that the results of Section~\ref{S:boundary} together with  Theorem~\ref{T:capping} give a proof of the main result Theorem~\ref{T:admissible}

\section{Minimum genus for embedding systolic graphs}

In this section, we prove the following:
\begin{theorem}\label{sec7thm1}
For each $g\geq 2$, there exists a closed hyperbolic surface $S_g$ such that the systolic graph $SLG(S_g)$ cannot be realized as a systolic graph of a closed surface of genus less than $g$.
\end{theorem}

Our proof is based on the existence of hyperbolic surfaces $S_g$ of genus $g$ with a filling set of systoles. This is well-known since the work of Schmitz~\cite{PSS} that any critical point of the function $syst$ has a filling set of systoles.

\begin{proof}[Proof of the Theorem~\ref{sec7thm1}]
Let $S_g$ be a closed hyperbolic surface such that the systolic graph $SLG(S_g)$ fills $S_g$. Then we show that the fat graph $SLG(S_g)$ cannot be realized as a systolic graph of any hyperbolic surface of genus less than $g$.

Let $F$ be a closed hyperbolic surface of genus less than $g$ such that the systolic graph $SLG(F)$ is isomorphic to $SLG(S_g)$. We have $S_g=SLG(S_g)\bigcup\limits_{i=1}^n D_i$ where $D_i$'s are discs and $n$ is the number of components in $S_g-SLG(S_g)$, the complement of $SLG(S_g)$ in $S_g$. Therefore, $$\chi(S_g)=\chi(SLG(S_g))+n.$$ Now, let $F_i, i=1,2, \ldots, k $ be the connected components in $F-SLG(F)$ then $(k\leq n)$ and we have, $$\chi(F) = \chi(SLG(F)) + \sum\limits_{i=1}^{k} \chi(F_i).$$ Each surface $F_i$ satisfies $\chi(F_i)\leq 1$. Thus we have $\sum\limits_{i=1}^{k}\chi(F_i)\leq n$ which is same as $\chi(SLG(S_g)) \leq \chi(F).$ Moreover, we have equality if and only if $n=k$ and each $F_i$  is a disc. In that case $F$ is isometric to $S_g$. Hence the result follows.
\end{proof}
We remark that the above result is based on a topological obstruction for a fat graph being embedded in a low genus surface. It would be interesting to know if there are geometric obstruction, i.e, admissible fat graphs that topologically embed in a surface but cannot be the systolic graph of the surface. Such a result may be based on a lower bound on the shortening of lengths of cycles under rounding, given an upper bound on the injectivity radius (in contrast to our main result being based on an upper bound on shortening of length due to rounding), with admissibility shown using computational tools. We hope to address this question in the future.

\end{document}